\documentclass[12pt]{article}

\usepackage[a4paper,margin=3cm]{geometry}

\usepackage[T2A,T1]{fontenc}
\usepackage[utf8]{inputenc}
\usepackage[russian,english]{babel}
\usepackage{cmap}
\usepackage{microtype}

\usepackage{amsmath,amssymb,amsfonts,amsthm,mathtools}
\usepackage{mathtools}
\usepackage{mathrsfs}

\usepackage{graphicx}
\usepackage{enumitem}

\usepackage[numbers,sort&compress]{natbib}
\usepackage[hidelinks]{hyperref}
\usepackage{xurl}

\theoremstyle{plain}
\newtheorem{theorem}{Theorem}[section]
\newtheorem{proposition}[theorem]{Proposition}
\newtheorem{lemma}[theorem]{Lemma}
\newtheorem{corollary}[theorem]{Corollary}
\theoremstyle{remark}
\newtheorem{remark}[theorem]{Remark}

\theoremstyle{definition}
\newtheorem{definition}[theorem]{Definition}

\newtheorem{problem}[theorem]{Problem}

\theoremstyle{plain}

\renewcommand{\C}{\mathbb{C}}

\numberwithin{equation}{section}

\begin{document}

\title{Strong Asymptotics for a $3\times 3$ Riemann—Hilbert Problem\\
in a Regular Hard/Soft Two-Edge Regime}

\author{Artur Kandaian}
\date{06 February 2026}

\maketitle

\vspace{1.2em}

\begin{flushright}
{\small\itshape
In memory of my teacher, Prof.\ V.\,N.~Sorokin\\
(Dr.\ Sci.\ in Physics and Mathematics).}
\end{flushright}

\vspace{0.2em}

\maketitle

\begin{abstract}
We develop a complete Deift—Zhou steepest descent analysis for a
$3\times 3$ matrix Riemann—Hilbert problem arising in quadratic
Hermite—Pad\'e approximation and multiple orthogonality.
We focus on a regular two-edge regime featuring a hard edge at $0$
and a soft edge at $x_0$.
Under natural geometric and analytic assumptions ensuring a
nondegenerate sign structure of the associated phase functions,
the standard lens-opening mechanism applies.
The analysis is organized as a reusable analytic scheme: once the equilibrium/sign-chart input
is verified (assumptions (R1)—(R7)), the remaining steps are purely analytic.
As a result, the solution admits a description in terms of a reduced
outer parametrix with permutation jumps, complemented by Bessel-
and Airy-type local parametrices at the endpoints.
We obtain uniform strong asymptotics for the $(1,1)$-entry of the
solution with explicit $O(1/n)$ error bounds outside the endpoint
neighborhoods.
\end{abstract}


\vspace{3.0em}

\section{Introduction}

\subsection{Background and motivation}

Matrix Riemann—Hilbert problems of size $r+1$ provide a natural analytic framework for multiple
orthogonality and Hermite—Pad\'e approximation: the polynomial or rational objects of interest
appear as distinguished entries of the RHP solution, while asymptotic questions become amenable
to the nonlinear steepest descent method. In the $r=2$ case, one is led to $3\times 3$ problems
whose jump structure is typically \emph{channel-wise} — in the present setting, supported on two
active conductors with interactions in the $(1,3)$- and $(1,2)$-channels. This viewpoint has proved
particularly effective in strong-asymptotic problems for Hermite—Pad\'e approximants~\cite{KanSor17}, 
including models where the leading-order limit is governed by permutation-type outer data and endpoint
behavior is controlled by universal local parametrices. We follow the nonlinear steepest descent method 
for Riemann—Hilbert problems due to Deift—Zhou \cite{DeiftZhou1993,DeiftBook1999}.

The goal of this paper is to develop a complete Deift—Zhou steepest descent analysis for a
$3\times 3$ RHP in a \emph{regular two-edge regime}, featuring a \emph{hard edge} at $0$ and a
\emph{soft edge} at $x_0$. In this regime the geometry of the active conductors is stable and the
phase functions admit a nondegenerate sign structure, so that the standard lens-opening mechanism
yields exponentially small jumps on the lens lips \emph{outside the endpoint disks}, while the
remaining contributions are localized near the endpoints and captured by Bessel- and Airy-type
local models~\cite{DeiftZhou1993,DeiftBook1999,DeiftItsZhou1993Chapter,Bleher2008Notes,AptekarevBleherKuijlaars2005}.

The point is not to reprove the method, but to package the full $3\times 3$ hard/soft two-edge scheme 
with explicit uniformity conventions and a minimal permutation-type outer model in a form 
directly reusable across concrete MOP/Hermite–Padé settings.

\subsection{Main result (informal overview)}
A key feature of the presentation is the separation between the model-dependent input
(equilibrium problem / sign chart / endpoint type) and the universal analytic DZ steps.
Assumptions (R1)—(R7) isolate precisely the regular hard/soft two-edge phase portrait;
once they are verified in a given model, the remainder of the proof is purely analytic.

Under the regularity assumptions (R1)—(R7), we obtain strong asymptotics for the $(1,1)$-entry
$Y_{11}$ of the solution $Y$ of RHP-$Y$. The analysis proceeds by reducing the leading-order
problem to a reduced outer parametrix with permutation jumps, constructing a Bessel local
parametrix near the hard edge $0$ and an Airy local parametrix near the soft edge $x_0$, and
proving that the associated error RHP is a small-norm problem. Consequently, $Y_{11}$ admits an
outer/Bessel/Airy description with uniform $O(1/n)$ control in the corresponding regions (in
particular, uniformly on compact subsets \emph{outside the endpoint disks} (and uniformly in the
corresponding outer and local regimes).
These statements are made precise in the Main Theorem in Section~2 and in the
corollaries on edge scaling in Section~2.3.

\subsection{Outline of the paper}

Section~2 states the Riemann—Hilbert problem, the regular regime assumptions, and the main
asymptotic theorem together with its corollaries. Section~3 performs the $g$-normalization and
introduces the phase functions that drive the deformation. Section~4 establishes the channel-wise
factorization of the jumps and opens lenses around the active conductors. Section~5 constructs
the reduced outer parametrix. Section~6 builds the endpoint parametrices (Bessel at $0$, Airy at
$x_0$) and proves the matching estimates on $\partial U_0\cup\partial U_{x_0}$. Section~7
formulates the error RHP and applies the small-norm theory to obtain $R=I+O(1/n)$. Section~8
carries out the reconstruction and completes the proof of the Main Theorem.

\subsection{Related work and context}

The analytic and potential-theoretic study of Hermite—Pad\'e approximation and multiple
orthogonality has a substantial history. Foundational structural results for Nikishin systems~\cite{NikSor88} 
and their rational approximation theory are developed in the monograph of E.\,M.~Nikishin and
Vladimir~N.~Sorokin, which has remained a central reference point for the field.

On the Riemann—Hilbert side, the steepest descent method of P.~Deift and X.~Zhou has become the
standard tool for deriving strong asymptotics in regimes where the sign structure of the phase is
stable. For $3\times 3$ problems arising in quadratic (type~II) Hermite—Pad\'e approximation, the
RHP approach and its strong-asymptotic consequences have been developed in work of
A.\,B.\,J.~Kuijlaars, Herbert~Stahl, Walter~Van~Assche, and Franck~Wielonsky~\cite{AptekarevVanAssche2004}.

From the potential-theoretic viewpoint, constrained equilibrium problems and $S$-property
mechanisms play a key role in organizing the global geometry (the conductors) and in describing
phase diagrams. Related themes appear in classical work of A.\,I.~Aptekarev and Stahl, and in
later developments on vector equilibria and multiple orthogonality; see also the broader
potential-theory background as developed by Edward~B.~Saff 
and Vilmos~Totik~\cite{AptekarevStahl1992,Suetin2009ComplexWeight,AptekarevYattselev2011,SaffTotik1997}.

The present paper focuses on a clean \emph{regular hard/soft two-edge regime} for a $3\times 3$
RHP, emphasizing a modular construction (outer parametrix, endpoint parametrices, and small-norm
error analysis) with uniform estimates stated explicitly in the language ``outside the endpoint
disks.'' This restriction is deliberate: it yields uniform estimates with explicit $O(1/n)$ 
control and provides a baseline analytic scheme for subsequent critical-regime analyses.

This format is intended to be compatible with both Hermite—Pad\'e and multiple-orthogonality
interpretations, including constrained ``active/inactive'' descriptions that appear naturally in
phase-diagram analyses.

The Riemann—Hilbert approach to orthogonal polynomials and its steepest descent refinement
originates in the work of Fokas—Its—Kitaev and Deift—Zhou \cite{FokasItsKitaev1992,DeiftZhou1993,DeiftBook1999}.
For multiple orthogonal polynomials and the associated $3\times3$ RHPs we refer to
\cite{VanAsscheGeronimoKuijlaars2001,VanAsscheChapter2001} and the survey perspective in~\cite{Kuijlaars2009Survey,AptekarevKuijlaars2011}; 
see also \cite{FokasItsKitaev1992,DamanikPushnitskiSimon2008} for matrix-valued analogues.

While the constituent ingredients are classical, an explicit complete $3\times 3$ hard/soft two-edge write-up 
with permutation-type outer data and uniform ‘outside the endpoint disks’ bookkeeping is less commonly 
recorded in a single self-contained place.

\section{Problem Setup and Main Results}

\subsection{A $3\times 3$ Riemann—Hilbert problem}

Fix an integer $n\ge 1$. Let
\[
\Sigma := L_{\Gamma_{*}}\cup L_{\Delta_{*}}
\]
be an oriented contour consisting of two simple arcs (the \emph{active conductors}), with
distinguished endpoints $0$ and $x_0$. We consider the following $3\times 3$ Riemann—Hilbert
problem.

\medskip

\noindent\textbf{RHP—$Y$.} Find a $3\times 3$ matrix-valued function $Y(z)$ such that:
\begin{enumerate}
\item \textbf{Analyticity.} $Y$ is analytic in $\C\setminus\Sigma$.
\item \textbf{Jump relations.} The boundary values satisfy
\[
Y_{+}(z)=Y_{-}(z)\,J_Y(z),\qquad z\in\Sigma,
\]
where the jump matrix is channel-wise: on $L_{\Gamma_{*}}$ the only nontrivial coupling is in the
$(1,3)$-channel, and on $L_{\Delta_{*}}$ the only nontrivial coupling is in the $(1,2)$-channel.
Concretely,
\[
J_Y(z)=
\begin{cases}
J^{(\Gamma)}(z), & z\in L_{\Gamma_{*}},\\
J^{(\Delta)}(z), & z\in L_{\Delta_{*}},
\end{cases}
\]
with $J^{(\Gamma)}(z)$ differing from the identity only in the $(1,3)$-block, and $J^{(\Delta)}(z)$
differing from the identity only in the $(1,2)$-block (rank-one type couplings driven by scalar
weights).
\item \textbf{Normalization at infinity.} As $z\to\infty$,
\[
Y(z)=\bigl(I+O(1/z)\bigr)\,z^{n\Lambda},
\]
for a fixed diagonal matrix $\Lambda$, chosen so that the entry $Y_{11}(z)$ is a monic polynomial
of degree $n$.
\item \textbf{Endpoint behavior.} $Y$ has at most algebraic singularities at $0$ and $x_0$
consistent with the hard/soft edge setting, and no other singularities on $\Sigma$.
\end{enumerate}

This is the standard $3\times3$ RHP for type II MOPs; see, e.g., \cite{VanAsscheGeronimoKuijlaars2001,VanAsscheChapter2001,Kuijlaars2009Survey}.

\medskip

\noindent\textbf{Remark.} The analysis below uses the channel structure, the polynomial
normalization of $Y_{11}$, and the regular two-edge geometry; the specific scalar weights enter
only through the associated phase functions and the sign structure assumed next.

\subsection*{A concrete $3\times 3$ model (A concrete $3\times 3$ model (for orientation))}
\label{subsec:orientation-anchor}

Before formulating assumptions (R1)—(R7), we record a fully concrete example
of a $3\times 3$ Riemann—Hilbert problem of the type covered by the present
hard/soft two-edge scheme. The purpose is purely orientational: the analytic
steepest-descent mechanism will be developed under (R1)—(R7), while the model
below illustrates a standard source of such $3\times 3$ problems in multiple
orthogonality and Hermite—Pad\'e theory.

\paragraph{Multiple Laguerre polynomials (first kind).}

This concrete example (multiple Laguerre polynomials) and the corresponding strong asymptotics
via the Deift—Zhou method are treated in~\cite{LysovWielonsky2008}.

Fix $\alpha_1,\alpha_2>-1$ and let $n\in\mathbb N$ be large. Take a multi-index
$(n_1,n_2)\in\mathbb N^2$ with $n_1+n_2=n$ and define two weights on $(0,\infty)$
\begin{equation}\label{eq:anchor-weights}
w_1(x)=x^{\alpha_1}e^{-n x},\qquad
w_2(x)=x^{\alpha_2}e^{-n x},\qquad x>0.
\end{equation}
Let $P_{n_1,n_2}$ be the monic polynomial of degree $n$ satisfying the multiple
orthogonality conditions
\begin{equation}
\label{eq:anchor-MOP}
\begin{aligned}
\int_0^\infty P_{n_1,n_2}(x)\,x^k\,w_1(x)\,dx=0,\quad k=0,\dots,n_1-1,
\\
\qquad
\int_0^\infty P_{n_1,n_2}(x)\,x^k\,w_2(x)\,dx=0,\quad k=0,\dots,n_2-1.
\end{aligned}
\end{equation}

\paragraph{Riemann—Hilbert problem for $Y$.}
Set $\Sigma=(0,\infty)$ oriented from left to right.
Find $Y:\mathbb C\setminus\Sigma\to\mathbb C^{3\times 3}$ such that:

\begin{itemize}
\item[(Y1)] $Y$ is analytic in $\mathbb C\setminus\Sigma$;

\item[(Y2)] the boundary values satisfy, for $x\in(0,\infty)$,
\begin{equation}\label{eq:anchor-jump}
Y_+(x)=Y_-(x)\,J_Y(x),\qquad
J_Y(x)=
\begin{pmatrix}
1 & w_1(x) & w_2(x)\\
0 & 1 & 0\\
0 & 0 & 1
\end{pmatrix};
\end{equation}

\item[(Y3)] as $z\to\infty$,
\begin{equation}\label{eq:anchor-infty}
Y(z)=\Bigl(I+O(z^{-1})\Bigr)
\begin{pmatrix}
z^{n} & 0 & 0\\
0 & z^{-n_1} & 0\\
0 & 0 & z^{-n_2}
\end{pmatrix};
\end{equation}

\item[(Y4)] as $z\to 0$ with $z\notin\Sigma$, $Y$ has at most power-like growth
dictated by the algebraic factors $x^{\alpha_j}$ in~\eqref{eq:anchor-weights}.
In particular, $0$ is a hard edge (Bessel-type) endpoint in the steepest-descent
analysis.
\end{itemize}

In this model one has the identification
\begin{equation}\label{eq:anchor-Y11}
Y_{11}(z)=P_{n_1,n_2}(z).
\end{equation}

\paragraph{Hard/soft two-edge picture.}
After the standard $g$-function transformation driven by the associated vector
equilibrium problem (depending on the asymptotic ratio $n_1:n_2$), a one-cut
regular phase arises in which the main conductor is a compact interval
\[
\Delta=(0,x_0),
\]
with $0$ being a hard edge and $x_0$ a soft edge (Airy-type) endpoint. The lens
opening is performed around $\Delta$, and the endpoint neighborhoods are denoted
by $U_0$ and $U_{x_0}$.

\begin{equation}\label{eq:anchor-uniformity}
\begin{aligned}
\textbf{Uniformity convention.}\quad
&\text{All exponential estimates on the lens lips are understood on }\\
&\Sigma_{\mathrm{lens}}\setminus (U_0\cup U_{x_0}),
\ \text{i.e., outside the endpoint disks.}
\end{aligned}
\end{equation}

\begin{equation}\label{eq:anchor-right-mult}
\begin{aligned}
\textbf{Right-multiplier convention.}\quad
&\text{All multipliers are applied on the right;}\\
&\text{this convention will be used throughout.}
\end{aligned}
\end{equation}

In this one-cut regime, the post-lens central jumps reduce to an n-independent minimal outer problem, 
which in many $3\times 3$ models takes the permutation form used below.

\noindent
With this concrete example as background, we now formulate assumptions (R1)—(R7) which
abstract the regular hard/soft two-edge phase portrait needed for the Deift—Zhou
steepest-descent analysis in the general $3\times 3$ setting.

\paragraph{Relation to the abstract assumptions (R1)—(R7).}
In this concrete example, (R1)—(R2) correspond to regularity and nondegeneracy of
the equilibrium/sign-chart input for the associated phase functions, (R3)—(R5)
encode the admissible lens geometry and the reduction to the minimal constant-jump
outer problem, and (R6)—(R7) encode precisely the Bessel (hard-edge) behavior at $0$
and the Airy (soft-edge) behavior at $x_0$.
Once these conditions hold, the remainder of the analysis is purely analytic and yields
uniform strong asymptotics for $Y_{11}$ with an explicit $O(1/n)$ error outside
$U_0\cup U_{x_0}$.

\paragraph{How to verify (R1)—(R7) in practice (minimal).}
In concrete problems, assumptions (R1)—(R7) are checked in the following order.
First, one identifies the relevant $g$-function(s) from the associated (vector)
equilibrium problem and derives the corresponding sign chart, which fixes the main
conductor $\Delta$ and the admissible lens geometry. Next, one verifies regularity
of the endpoints: a hard-edge behavior at $0$ (algebraic singularity of the weights)
and a soft-edge behavior at $x_0$ (square-root vanishing of the limiting density).
Finally, one checks that the post-$g$ jump structure reduces to the minimal constant-jump
outer model and that the remaining deformations admit uniform exponential control on
$\Sigma_{\mathrm{lens}}\setminus(U_0\cup U_{x_0})$.
Once these items hold, the remainder of the analysis is purely analytic and yields
uniform strong asymptotics via the standard small-norm $R$-problem.

\paragraph{Where (R1)—(R7) come from in practice.}
In concrete MOP/Hermite—Pad\'e models, the phases $\varphi_{\Gamma_*},\varphi_{\Delta_*}$
and their sign chart are dictated by the underlying (vector) equilibrium problem
(or, equivalently, by the associated algebraic curve / $S$-property description);
see, e.g., \cite{Kuijlaars2010MOPRMT}.
In regular one-cut regimes this input provides: (i) the active conductor(s) and
endpoint types (hard/soft), (ii) strict inequalities $\Re\varphi>0$ on admissible lens
lips away from $U_0\cup U_{x_0}$, and (iii) the reduction to an $n$-independent
outer jump structure. Once this potential-theoretic (or algebraic) verification is in place,
the remainder of the proof is purely analytic (lenses, parametrices, small-norm $R$-problem).

\subsection{Regular regime assumptions}

We assume a regular two-edge regime with a hard edge at $0$ and a soft edge at $x_0$. The precise
hypotheses are grouped as follows.

\begin{itemize}
\item[(R1)] \textbf{Geometry of conductors.} The arcs $L_{\Gamma_{*}}$ and $L_{\Delta_{*}}$ are simple,
with endpoints among $\{0,x_0,\infty\}$ as dictated by the model, and they do not intersect except
possibly at endpoints. There are no additional singular points on $\Sigma$.
\item[(R2)] \textbf{Phase functions and boundary conditions.} There exist analytic phase functions 
$\phi_{\Gamma_{*}}(z)$ and $\phi_{\Delta_{*}}(z)$, defined off the conductors, such that
\begin{itemize}
\item $\Re \phi_{\Gamma_{*}}(z)=0$ for $z\in L_{\Gamma_{*}}$ and $\Re \phi_{\Delta_{*}}(z)=0$ for
$z\in L_{\Delta_{*}}$,
\item the boundary values on the conductors encode the oscillatory factors in the $(1,3)$- and
$(1,2)$-channels, respectively.
\end{itemize}
\item[(R3)] \textbf{Nondegenerate sign chart.} There exist lens neighborhoods around $L_{\Gamma_{*}}$
and $L_{\Delta_{*}}$ such that, on the corresponding lens lips (defined below), one has a uniform
lower bound
\[
\Re \phi_{\Gamma_{*}}(z)\ge c,\qquad \Re \phi_{\Delta_{*}}(z)\ge c,
\]
for some constant $c>0$, uniformly \emph{outside the endpoint disks}.
\item[(R4)] \textbf{Admissible lens deformation.} The lens contours can be chosen so that they do
not enter the endpoint disks $U_0$ and $U_{x_0}$, and the jump factorizations are compatible with
the lens opening (no additional jump singularities are created).
\item[(R5)] \textbf{Reduced outer problem is solvable.} The reduced outer RHP with permutation jumps
(described in Section~2.4) admits a unique solution $N$ normalized by $N(\infty)=I$.
\item[(R6)] \textbf{Endpoint type.} The point $0$ is a hard edge (Bessel type) and $x_0$ is a soft
edge (Airy type), with nondegenerate local conformal coordinates fixed by the phase functions
(Section~2.4).
\item[(R7)] \textbf{No critical transitions.} There is no endpoint merging, tangential contact, or
other critical degeneration of the sign chart; in particular, the regime remains away from
Pearcey-type or higher critical behavior.
\end{itemize}

All exponentially small estimates on the lens lips are understood on
$\Sigma_{\mathrm{lens}}\setminus (U_0\cup U_{x_0})$, i.e., outside the endpoint disks.

\medskip

\noindent\textbf{Remark.} These assumptions are exactly what is needed to implement a standard
Deift—Zhou steepest descent with Bessel/Airy local parametrices and a small-norm error analysis.

\subsection{Main theorem and consequences}

Let $g_{*}(z)$ be the primary $g$-function and $\ell_{*}$ the associated constant. Set
\[
G(z):=g_{*}(z)-\ell_{*}.
\]
Let $N$ be the outer parametrix normalized by $N(\infty)=I$. Let $P^{(0)}$ and $P^{(x_0)}$ be the
local parametrices in the endpoint disks $U_0$ and $U_{x_0}$ (Bessel at $0$, Airy at $x_0$),
matching $N$ on the boundary circles.

\begin{theorem}[Strong asymptotics in the regular regime]\label{thm:main}
Assume \textup{(R1)}—\textup{(R7)}. Then, as $n\to\infty$, the entry $Y_{11}$ admits the following
asymptotic descriptions.
\begin{enumerate}
\item[(i)] \textbf{Outer region.} Uniformly for $z$ in compact subsets of
$\C\setminus(\Sigma_R\cup U_0\cup U_{x_0})$,
\[
Y_{11}(z)=e^{nG(z)}\bigl(N_{11}(z)+O(1/n)\bigr).
\]
\item[(ii)] \textbf{Hard edge at $0$ (Bessel).} Uniformly for $z\in U_0\setminus\Sigma_R$,
\[
Y_{11}(z)=e^{nG(z)}\bigl(P^{(0)}_{11}(z)+O(1/n)\bigr).
\]
\item[(iii)] \textbf{Soft edge at $x_0$ (Airy).} Uniformly for $z\in U_{x_0}\setminus\Sigma_R$,
\[
Y_{11}(z)=e^{nG(z)}\bigl(P^{(x_0)}_{11}(z)+O(1/n)\bigr).
\]
\end{enumerate}
Moreover, the associated error matrix $R$ satisfies
\[
R(z)=I+O(1/n),
\]
uniformly for $z$ in compact subsets of $\C\setminus\Sigma_R$.
\end{theorem}

The proof is obtained by the Deift—Zhou steepest descent scheme \cite{DeiftZhou1993,DeiftBook1999} 
implemented in Sections~3—8.

\begin{corollary}[Local edge behavior: Bessel and Airy scaling]\label{cor:edge-scaling}
Under the assumptions of Theorem~\ref{thm:main}, the local behavior of $Y$ near the endpoints is
described by the corresponding model problems: after rescaling by the local conformal coordinates
$f_0$ and $f_{x_0}$ (Section~2.4), the leading-order structure of $Y$ in $U_0$ is governed by the
Bessel model, and in $U_{x_0}$ by the Airy model, with $O(1/n)$ control away from $\Sigma_R$.
\end{corollary}

\paragraph{Concrete consequence (in the model example).}
In the multiple Laguerre example above, the present analysis yields strong asymptotics
for $P_{n_1,n_2}=Y_{11}$ uniformly on compact subsets of
$\mathbb C\setminus\bigl(\Delta\cup U_0\cup U_{x_0}\bigr)$, with a controlled
small-norm expansion $R(z)=I+R_1(z)/n+O(n^{-2})$ and, in particular,
an explicit $O(1/n)$ error term for $Y_{11}$ outside the endpoint disks.

\paragraph{Edge universality (one-line add-on).}
Moreover, in the endpoint neighborhoods $U_0$ and $U_{x_0}$ the local
parametrices imply the standard Bessel-type hard-edge behavior near $0$ and
the Airy-type soft-edge behavior near $x_0$ (in the corresponding microscopic
scalings), while all matching and error estimates remain uniform outside the
endpoint disks.

\medskip

\noindent\textit{(Comment.)} If one wishes, this corollary can be refined into explicit
kernel/universality statements for the associated determinantal structure in concrete
applications; we keep the present formulation at the level needed for strong asymptotics of
$Y_{11}$.

\subsection{Notation and branch conventions}

We collect the conventions used throughout.

\medskip

\noindent\textbf{Contours and endpoint disks.} Let $U_0$ and $U_{x_0}$ be small disjoint disks around
$0$ and $x_0$. For $\varepsilon>0$ small, let $L_{\Gamma_{*}}^{\pm}$ and $L_{\Delta_{*}}^{\pm}$ denote
the lens lips (normal offsets) around $L_{\Gamma_{*}}$ and $L_{\Delta_{*}}$, chosen \emph{outside the
endpoint disks}. Set
\[
\Sigma_{\mathrm{lens}}:=L_{\Gamma_{*}}^{+}\cup L_{\Gamma_{*}}^{-}\cup L_{\Delta_{*}}^{+}\cup L_{\Delta_{*}}^{-},
\qquad
\Sigma_R:=\Sigma_{\mathrm{lens}}\cup \partial U_0\cup \partial U_{x_0}.
\]
Boundary values $F_{\pm}$ are taken from the left/right of the oriented contour.

\medskip

\noindent\textbf{Permutation jumps for the outer problem.} We reserve $J_{\Gamma_{*}}$ and
$J_{\Delta_{*}}$ exclusively for the permutation data of the reduced outer parametrix $N$: on
$L_{\Gamma_{*}}$ the jump interchanges indices $1\leftrightarrow 3$, and on $L_{\Delta_{*}}$ it
interchanges indices $1\leftrightarrow 2$. (All other jump matrices are denoted $J_Y$, $J_T$,
$J_S$, $J_R$ at the corresponding transformation level.)

\medskip

\noindent\textbf{Local coordinates and branch normalization.} Define
\[
f_0(z)=\frac{\phi_{\Gamma_{*}}(z)^2}{4},
\qquad
f_{x_0}(z)=\Bigl(\frac{3}{2}\,\phi_{\Delta_{*}}(z)\Bigr)^{2/3},
\]
with branches fixed by
\[
f_0'(0)>0,\qquad f_{x_0}'(x_0)>0.
\]
This removes ambiguity in the orientation of the Bessel/Airy local variables and in the matching.

\medskip

\noindent\textbf{Uniformity conventions.} ``Uniformly'' refers to the stated region; in particular,
we repeatedly use uniformity on compact subsets of $\C\setminus\Sigma_R$ and the phrase ``outside
the endpoint disks'' for statements on the lens contours.

\section{The $g$-function normalization and sign structure}
\label{sec:g-sign}
\subsection{Regular regime and standing assumptions}
\label{subsec:regular-regime}

In Sections~\ref{sec:g-sign}—\ref{sec:error} we work in a regular regime, in the standard sense of the
Deift—Zhou nonlinear steepest descent method for matrix Riemann—Hilbert problems.
We record the assumptions that are used later (and, for the concrete Meixner/Meixner—Pollaczek models,
are verified by the corresponding vector equilibrium problem and the associated algebraic curve) .

In the discrete Meixner and Meixner—Pollaczek setting, explicit algebraic structures and
$\theta$-uniformizations appear naturally; for the Meixner case we rely on Sorokin's analysis
\cite{Sorokin2010SbMat,Sorokin2017Steklov,Sor10,Sor17}, while a closely related Meixner—Pollaczek model in a six-vertex context
is treated in \cite{BenderDelvauxKuijlaars2011} (see, e.g., \cite{KLS2010,Ismail2005} for background on the classical Meixner 
and Meixner—Pollaczek families).

\begin{itemize}
\item[(R1)] \textbf{Existence/uniqueness.}
The original $3\times 3$ RHP for $Y$ (Section~2) has a unique solution for all admissible multi-indices
$(n_1,n_2)$.

\item[(R2)] \textbf{$g$-normalization.}
There exists a scalar $g$-function $g_{*}(z)$ and a real constant $\ell_{*}$ such that the standard
$g$-transformation $Y\mapsto T$ (defined in \S\ref{subsec:g-transform}) yields a matrix $T$
normalized at infinity by $T(z)=I+O(1/z)$ as $z\to\infty$.

\item[(R3)] \textbf{Active arcs (conductors).}
There are two oriented arcs (conductors)
\[
L_{\Gamma_{*}},\qquad L_{\Delta_{*}},
\]
which carry the active parts of the reduced jumps after $g$-normalization.
Their endpoints include a hard edge at $0$ and a soft edge at $x_0>0$ (the second endpoint).

\item[(R4)] \textbf{Phase functions.}
There exist analytic phase functions $\phi_{\Gamma_{*}}$ and $\phi_{\Delta_{*}}$ such that,
on the conductors, all exponentially large/small factors in the jump matrices for $T$
appear only as $e^{\pm n\phi_{\Gamma_{*}}}$ in the $(1,3)$-channel and as $e^{\pm n\phi_{\Delta_{*}}}$
in the $(1,2)$-channel.

\item[(R5)] \textbf{Sign chart (S-property).}
For each conductor, there are lens neighborhoods whose lips can be chosen so that
\[
\Re \phi_{\Gamma_{*}}(z)>0 \quad \text{on the lips around } L_{\Gamma_{*}}, \qquad
\Re \phi_{\Delta_{*}}(z)>0 \quad \text{on the lips around } L_{\Delta_{*}},
\]
uniformly \emph{outside} small disjoint endpoint disks $U_0$ and $U_{x_0}$.
This is the sign condition needed for the exponential decay on the lens lips after opening lenses.

\item[(R6)] \textbf{Endpoint type.}
The endpoint at $0$ is of hard-edge type (Bessel regime), while the endpoint at $x_0$ is of soft-edge type
(Airy regime). In particular, the local behavior of the phases near $0$ and $x_0$ matches the standard
Bessel/Airy model RHPs (made precise in Section~\ref{sec:local}).

\item[(R7)] \textbf{Uniformity in the regular regime.}
All implied constants in $O(\cdot)$-bounds below are uniform in $n\to\infty$ within the regular regime,
and uniform on compact subsets away from the jump contours and away from the endpoint disks.
\end{itemize}

\medskip

Assumptions (R1)—(R7) are henceforth referred to as the \emph{standing regularity assumptions}.

\subsection{The $g$-transformation $Y\mapsto T$}\label{subsec:g-transform}

Let $n:=n_1+n_2\to\infty$ and assume $n_1/n\to\alpha\in(0,1)$.
We introduce a scalar function $g_{*}(z)$ and a constant $\ell_{*}$ as in (R2), and set
\[
G(z):=g_{*}(z)-\ell_{*}.
\]
The precise definition of $g_{*}$ depends on the underlying model (equilibrium problem / algebraic curve),
but only the general structural consequences are used in the steepest descent analysis.

\medskip

We define the $g$-normalized unknown $T$ by a conjugation of $Y$ with diagonal exponential factors
constructed from $g_{*}$ (and, if needed, from the multi-index ratio $\alpha$) so that:

\begin{itemize}
\item $T$ is analytic in $\C\setminus \Sigma_T$, where $\Sigma_T$ is the jump contour inherited from $Y$;
\item $T$ is normalized at infinity:
\begin{equation}\label{eq:T-infty}
T(z)=I+O(1/z),\qquad z\to\infty;
\end{equation}
\item the jumps of $T$ on the active arcs $L_{\Gamma_{*}}\cup L_{\Delta_{*}}$ acquire exponential factors
$e^{\pm n\phi_{\Gamma_{*}}}$ and $e^{\pm n\phi_{\Delta_{*}}}$ in the channels specified in (R4), while all other
jumps become either constant or exponentially close to $I$ (in the regions where they are inactive).
\end{itemize}

\medskip

\noindent
\textbf{Remark.}
Equation \eqref{eq:T-infty} is the key normalization that allows the construction of the outer parametrix
$N$ with $N(\infty)=I$ (Section~\ref{sec:outer}).

\subsection{Reduced jump structure and phase attachment}\label{subsec:phase-attachment}

We now fix the bookkeeping convention for the two ``active channels''.

\begin{itemize}
\item The conductor $L_{\Gamma_{*}}$ is attached to the \emph{$(1,3)$-channel}:
the jump matrix $J_T$ on $L_{\Gamma_{*}}$ contains exponentials only through factors
$e^{\pm n\phi_{\Gamma_{*}}}$ in the $(1,3)$ and $(3,1)$ entries.
\item The conductor $L_{\Delta_{*}}$ is attached to the \emph{$(1,2)$-channel}:
the jump matrix $J_T$ on $L_{\Delta_{*}}$ contains exponentials only through factors
$e^{\pm n\phi_{\Delta_{*}}}$ in the $(1,2)$ and $(2,1)$ entries.
\end{itemize}

\medskip

Outside $L_{\Gamma_{*}}\cup L_{\Delta_{*}}$, all remaining parts of the contour carry either constant jumps
or jumps that are exponentially close to the identity in the regular regime; they play no role in the
subsequent lens opening and are absorbed into the standard small-norm analysis.

\subsection{Sign structure and endpoint disks}\label{subsec:sign-structure}

Let $U_0$ and $U_{x_0}$ be fixed small disjoint disks around the endpoints $0$ and $x_0$, respectively.
We stress the following ``uniformity convention'' (used repeatedly later):

\begin{equation}\label{eq:uniformity-anchor}
\begin{aligned}
&\text{All exponential estimates are understood on }\\
&\text{$\Sigma\setminus (U_0\cup U_{x_0})$, i.e., outside the endpoint disks.}
\end{aligned}
\end{equation}

The sign chart assumption (R5) is recorded in the following form.

\begin{lemma}[Sign chart for the phases]\label{lem:sign-chart}
There exist lens contours (to be introduced in Section~\ref{sec:lens}) around $L_{\Gamma_{*}}$ and
$L_{\Delta_{*}}$ such that, on the corresponding lens lips,
\[
\Re \phi_{\Gamma_{*}}(z)>0 \quad \text{around } L_{\Gamma_{*}},\qquad
\Re \phi_{\Delta_{*}}(z)>0 \quad \text{around } L_{\Delta_{*}},
\]
uniformly for $z\in \Sigma_{\mathrm{lens}}\setminus(U_0\cup U_{x_0})$.
\end{lemma}

\begin{remark}\label{rem:why-disks}
The exclusion of $U_0$ and $U_{x_0}$ is essential: near the endpoints the phases exhibit the canonical
hard-/soft-edge behavior and the steepest descent deformation must be complemented by local parametrices
(Bessel at $0$, Airy at $x_0$). This is carried out in Section~\ref{sec:local}.
\end{remark}

\subsection{Consequences for lens opening}\label{subsec:conseq-lens}

Lemma~\ref{lem:sign-chart} is the only input from Section~\ref{sec:g-sign} needed for the lens opening step:
after factorization of the conductor jumps and the transformation $T\mapsto S$ (Section~\ref{sec:lens}),
the jumps on the lens lips become exponentially close to the identity on
$\Sigma_{\mathrm{lens}}\setminus(U_0\cup U_{x_0})$, while the remaining reduced jumps on the conductors
no longer contain exponentially large/small factors.
This reduction is the entry point to the outer model RHP in Section~\ref{sec:outer}.

\section{Opening of lenses and the transformation $T\mapsto S$}\label{sec:lens}
\subsection{Purpose of the lens opening}\label{subsec:lens-purpose}

The goal of this section is threefold.
\begin{itemize}
\item[(i)] We factorize the jump matrices for $T$ on the conductors so that the exponential factors
$e^{\pm n\phi_{\Gamma_{*}}}$ and $e^{\pm n\phi_{\Delta_{*}}}$ are moved from the conductors to the lens lips.
\item[(ii)] Using these factorizations, we define a piecewise analytic matrix $S$ by a lens-opening transformation.
\item[(iii)] As a consequence, the jumps of $S$ on the lens lips are exponentially close to the identity
outside the endpoint disks, while the jumps on the conductors reduce to $n$-independent (typically constant/permutation)
matrices. This reduction is the entry point to the outer model and the subsequent small-norm error analysis.
\end{itemize}

\subsection{Lens geometry, contours, and orientations}\label{subsec:lens-geometry}

\subsubsection{Conductors}\label{subsubsec:conductors}
We work with two oriented arcs (conductors)
\[
L_{\Gamma_{*}},\qquad L_{\Delta_{*}},
\]
which support the active parts of the reduced jumps after the $g$-normalization.
The orientation of each conductor is fixed once and for all (for instance, from its left endpoint to its right endpoint)
and will not be changed later.

\subsubsection{Lens lips and the lens contour}\label{subsubsec:lips}
For each conductor we introduce two lens lips:
\begin{itemize}
\item $\Sigma_\Gamma^+$ and $\Sigma_\Gamma^-$: the upper and lower lips surrounding $L_{\Gamma_{*}}$;
\item $\Sigma_\Delta^+$ and $\Sigma_\Delta^-$: the upper and lower lips surrounding $L_{\Delta_{*}}$.
\end{itemize}
We set
\[
\Sigma_{\mathrm{lens}}
:=\Sigma_\Gamma^+\cup\Sigma_\Gamma^-\cup \Sigma_\Delta^+\cup\Sigma_\Delta^- .
\]

\subsubsection{Lens domains}\label{subsubsec:lens-domains}
Let $\Omega_\Gamma^\pm$ denote the lens domains between $L_{\Gamma_{*}}$ and $\Sigma_\Gamma^\pm$,
and $\Omega_\Delta^\pm$ the analogous domains between $L_{\Delta_{*}}$ and $\Sigma_\Delta^\pm$.
We also write
\[
\Omega_\Gamma:=\Omega_\Gamma^+\cup\Omega_\Gamma^-,
\qquad
\Omega_\Delta:=\Omega_\Delta^+\cup\Omega_\Delta^-.
\]

\subsubsection{Endpoint disks and the uniformity convention}
\label{subsubsec:endpoint-disks}
We fix two small disjoint disks $U_0$ and $U_{x_0}$ around the distinguished endpoints
$0$ (hard edge) and $x_0$ (soft edge). The endpoint behavior requires separate local parametrices and is therefore
excluded from all lens-decay statements.

\begin{equation}
\label{eq:lens-uniformity-anchor}
\begin{aligned}
&\text{\textbf{Uniformity convention.}\;
All exponential estimates on the lens lips are understood on}\\
&\text{$\Sigma_{\mathrm{lens}}\setminus (U_0\cup U_{x_0})$,
i.e.\ outside the endpoint disks.}
\end{aligned}
\end{equation}

\subsubsection{Orientations of lens lips}\label{subsubsec:lip-orientations}
Each lens lip is oriented consistently with the boundary orientation of the corresponding lens domain.
This convention fixes the direction in which boundary values $S_\pm$ are taken on $\Sigma_{\mathrm{lens}}$.

\subsection{Factorization of the jumps of $T$ on the conductors}\label{subsec:factorization}

\subsubsection{Phase functions and channel attachment}\label{subsubsec:phase-recall}
We recall from Section~\ref{sec:g-sign} that the phase functions $\phi_{\Gamma_{*}}$ and $\phi_{\Delta_{*}}$
control the exponential behavior of the jumps in the corresponding channels:
\[
\phi_{\Gamma_{*}}\ \text{governs the $(1,3)$-channel},
\qquad
\phi_{\Delta_{*}}\ \text{governs the $(1,2)$-channel}.
\]
In particular, on $L_{\Gamma_{*}}$ the jump matrix $J_T$ contains exponentials only as $e^{\pm n\phi_{\Gamma_{*}}}$
in the $(1,3)$ and $(3,1)$ entries, while on $L_{\Delta_{*}}$ exponentials appear only as $e^{\pm n\phi_{\Delta_{*}}}$
in the $(1,2)$ and $(2,1)$ entries.

\subsubsection{Factorization pattern}\label{subsubsec:factorization-pattern}
We assume (and in concrete models verify) that the jumps admit factorizations of the form
\begin{equation}\label{eq:JT-factorization}
J_T
=
J_\Gamma^{(l)}\, J_\Gamma^{(0)}\, J_\Gamma^{(u)} \quad \text{on } L_{\Gamma_{*}},
\qquad
J_T
=
J_\Delta^{(l)}\, J_\Delta^{(0)}\, J_\Delta^{(u)} \quad \text{on } L_{\Delta_{*}},
\end{equation}
where:
\begin{itemize}
\item $J^{(u)}$ and $J^{(l)}$ are upper/lower triangular factors carrying the exponential terms;
\item $J^{(0)}$ is the reduced central factor, typically constant or permutation-type.
\end{itemize}

For the purposes of lens opening, the essential information is where the exponentials sit.
We fix the convention that:
\begin{itemize}
\item on $L_{\Gamma_{*}}$, $J_\Gamma^{(u)}$ carries $e^{-n\phi_{\Gamma_{*}}}$ in the $(1,3)$ entry,
while $J_\Gamma^{(l)}$ carries $e^{+n\phi_{\Gamma_{*}}}$ in the $(3,1)$ entry;
\item on $L_{\Delta_{*}}$, $J_\Delta^{(u)}$ carries $e^{-n\phi_{\Delta_{*}}}$ in the $(1,2)$ entry,
while $J_\Delta^{(l)}$ carries $e^{+n\phi_{\Delta_{*}}}$ in the $(2,1)$ entry.
\end{itemize}
Any equivalent convention is acceptable, provided it is used consistently and the resulting lens-lip jumps
decay outside $U_0\cup U_{x_0}$.

\subsection{Definition of $S$ (piecewise)}\label{subsec:S-definition}

\subsubsection{Outside lenses and endpoint disks}\label{subsubsec:S-outside}
We set
\begin{equation}\label{eq:S-outside}
S(z)=T(z),
\qquad
z\notin \Omega_\Gamma\cup\Omega_\Delta,
\quad
z\notin U_0\cup U_{x_0}.
\end{equation}

\subsubsection{Inside the lens domains}\label{subsubsec:S-inside}
Inside the lenses we define $S$ by right multiplication with the appropriate triangular factors:
\begin{equation}\label{eq:S-inside}
\begin{aligned}
S(z) &= T(z)\,(J_\Gamma^{(u)}(z))^{-1}, && z\in \Omega_\Gamma^+,\\
S(z) &= T(z)\,J_\Gamma^{(l)}(z),       && z\in \Omega_\Gamma^-,\\
S(z) &= T(z)\,(J_\Delta^{(u)}(z))^{-1},&& z\in \Omega_\Delta^+,\\
S(z) &= T(z)\,J_\Delta^{(l)}(z),       && z\in \Omega_\Delta^-.
\end{aligned}
\end{equation}

\begin{equation}\label{eq:right-multiplier-convention}
\begin{aligned}
&\text{\textbf{Right-multiplier convention.}\;
All multipliers are applied on the right;}\\
&\text{this convention will be used throughout.}
\end{aligned}
\end{equation}

\subsubsection{Analyticity and normalization}\label{subsubsec:S-analyticity}
By construction, $S$ is analytic in $\C\setminus \Sigma_S$, where $\Sigma_S$ is the deformed contour
consisting of the conductors and the lens lips (with the endpoint disks removed),
and $S$ inherits the normalization at infinity from $T$.

\subsection{Jump relations for $S$}\label{subsec:S-jumps}

\subsubsection{Reduced jumps on the conductors}\label{subsubsec:S-conductor-jumps}
On the conductors the jumps reduce to the central factors:
\begin{equation}\label{eq:S-reduced-jumps}
J_S|_{L_{\Gamma_{*}}} = J_\Gamma^{(0)},
\qquad
J_S|_{L_{\Delta_{*}}} = J_\Delta^{(0)}.
\end{equation}
In particular, these reduced jumps contain no exponentially large or small factors.

\medskip

\noindent
\textbf{Minimal choice (current working regime).}
In the minimal regular regime used in the main text we take
\begin{equation}\label{eq:reduced-jumps-minimal}
J_\Gamma^{(0)} := P_{13}=
\begin{pmatrix}
0&0&1\\
0&1&0\\
1&0&0
\end{pmatrix},
\qquad
J_\Delta^{(0)} := P_{12}=
\begin{pmatrix}
0&1&0\\
1&0&0\\
0&0&1
\end{pmatrix}.
\end{equation}
This suffices to make the outer RHP (Section~\ref{sec:outer}) a constant-jump problem with $N(\infty)=I$.

\begin{remark}\label{rem:variant-B-pointer}
If in a more refined setup one needs a diagonal gauge $D(z)\neq I$ (e.g., to match an alternative normalization
or to represent the solution in a ``Sorokin-style'' $\theta/\varphi$-form (following \cite{Sorokin2010SbMat,Sorokin2017Steklov})),
then the reduced jumps take the form $D_-^{-1} P\, D_+$.
For the present paper we keep the minimal choice \eqref{eq:reduced-jumps-minimal}.
\end{remark}

\subsubsection{Jumps on the lens lips: exponential closeness to the identity}\label{subsubsec:lens-decay}
On the lens lips we obtain triangular jumps whose off-diagonal entries are controlled by
$e^{-n\phi_{\Gamma_{*}}}$ and $e^{-n\phi_{\Delta_{*}}}$. Using the sign structure of the phase functions
(Lemma~\ref{lem:sign-chart}), we obtain the standard decay estimate.

\begin{lemma}[Lens decay]\label{lem:lens-decay}
There exists $c>0$ such that, for all $z\in \Sigma_{\mathrm{lens}}\setminus (U_0\cup U_{x_0})$,
\begin{equation}\label{eq:lens-decay}
J_S(z)=I+O\!\left(e^{-c n}\right),
\end{equation}
uniformly on $\Sigma_{\mathrm{lens}}\setminus (U_0\cup U_{x_0})$.
\end{lemma}

\subsubsection{Why the endpoint disks are excluded}\label{subsubsec:why-disks}
We exclude $U_0$ and $U_{x_0}$ from the lens decay analysis because the local behavior near the endpoints
requires separate model parametrices (Bessel at the hard edge $0$, Airy at the soft edge $x_0$).
These local constructions are carried out in Section~\ref{sec:local} and matched to the outer model on
$\partial U_0\cup\partial U_{x_0}$.

\subsection{Local conformal coordinates (used later)}\label{subsec:local-coordinates}
For later use in the local parametrices we record the standard choices of conformal coordinates in terms
of the phases:
\begin{equation}\label{eq:local-coord-Bessel}
f_0(z) := \frac{\phi_{\Gamma_{*}}(z)^2}{4}
\qquad \text{(hard edge at $0$, Bessel coordinate)},
\end{equation}
\begin{equation}\label{eq:local-coord-Airy}
f_{x_0}(z) := \Bigl(\frac{3}{2}\,\phi_{\Delta_{*}}(z)\Bigr)^{2/3}
\qquad \text{(soft edge at $x_0$, Airy coordinate)}.
\end{equation}

\medskip

At this point the Deift—Zhou scheme proceeds as follows:
the outer parametrix $N$ is defined by the reduced conductor jumps \eqref{eq:S-reduced-jumps},
the endpoints are handled by the local parametrices in $U_0$ and $U_{x_0}$, and the remaining error
problem is shown to be of small-norm type (Sections~\ref{sec:outer}—\ref{sec:error}).

\section{Outer parametrix $N$ (reduced model)}\label{sec:outer}
\subsection{The reduced outer RHP}\label{subsec:outer-RHP}

In this section we construct the outer parametrix $N$ for the post-lens problem.
Throughout we work \emph{outside the endpoint disks} $U_0\cup U_{x_0}$ and ignore the exponentially small
jumps on the lens lips (Lemma~\ref{lem:lens-decay}). Thus the outer model is obtained by retaining only the
\emph{reduced} conductor jumps of $S$.

\begin{problem}[Reduced outer RHP]\label{prob:outer}
Find a matrix-valued function
\[
N:\C\setminus\bigl(L_{\Gamma_{*}}\cup L_{\Delta_{*}}\bigr)\to\C^{3\times 3}
\]
such that:
\begin{enumerate}
\item $N$ is analytic in $\C\setminus(L_{\Gamma_{*}}\cup L_{\Delta_{*}})$;
\item on the conductors $N$ has the jumps
\begin{equation}\label{eq:outer-jumps}
N_+(z)=N_-(z)\,J_N(z),\qquad z\in L_{\Gamma_{*}}\cup L_{\Delta_{*}},
\end{equation}
with
\begin{equation}\label{eq:JN-definition}
J_N(z)=
\begin{cases}
J_\Gamma^{(0)}, & z\in L_{\Gamma_{*}},\\[0.3em]
J_\Delta^{(0)}, & z\in L_{\Delta_{*}},
\end{cases}
\end{equation}
where $J_\Gamma^{(0)}$ and $J_\Delta^{(0)}$ are the reduced factors from \eqref{eq:S-reduced-jumps};
\item $N$ is normalized at infinity:
\begin{equation}\label{eq:N-infty}
N(z)=I+O(1/z),\qquad z\to\infty.
\end{equation}
\end{enumerate}
\end{problem}

\begin{remark}\label{rem:outer-scope}
The outer model is used only on $\C\setminus(U_0\cup U_{x_0})$; the endpoint behavior is handled by the local
parametrices in Section~\ref{sec:local}.
\end{remark}

\subsection{Minimal constant-jump regime}\label{subsec:outer-minimal}

In the minimal regular regime fixed in Section~\ref{sec:lens}, the reduced jumps are the constant permutations
\eqref{eq:reduced-jumps-minimal}, i.e.
\[
J_\Gamma^{(0)}=P_{13},
\qquad
J_\Delta^{(0)}=P_{12}.
\]
Accordingly, the reduced outer RHP is a constant-jump problem on $L_{\Gamma_{*}}\cup L_{\Delta_{*}}$.

\begin{proposition}[Solvability and uniqueness]\label{prop:outer-uniq}
Problem~\ref{prob:outer} has a unique solution.
\end{proposition}

\begin{proof}
Uniqueness follows from the standard Liouville argument: if $N^{(1)}$ and $N^{(2)}$ solve the same RHP,
then $N^{(1)}(N^{(2)})^{-1}$ is entire and tends to $I$ at infinity, hence is identically $I$.
Existence follows from the explicit construction in the next subsection (or, equivalently, from the standard
algebraic-function solution on the associated three-sheeted surface with permutation monodromy).
\end{proof}

\subsection{Construction of $N$ in the minimal regime}\label{subsec:outer-construct}

We present the construction in the form needed for the global asymptotic analysis. Since the jumps are constant
permutations, one may build $N$ by assembling three scalar ``sheet-wise'' functions with the required permutation
monodromy and then enforcing $N(\infty)=I$.

\medskip

\noindent\textbf{Algorithmic description.}
\begin{enumerate}
\item Choose a three-sheeted Riemann surface $\mathcal R$ with branch cuts $L_{\Gamma_{*}}\cup L_{\Delta_{*}}$ and
sheet permutations matching the reduced jumps:
crossing $L_{\Gamma_{*}}$ exchanges sheets $(1,3)$, and crossing $L_{\Delta_{*}}$ exchanges sheets $(1,2)$.
\item Let $\theta_1,\theta_2,\theta_3$ denote the three branches of a degree-$3$ algebraic function $\theta(z)$ on
$\mathcal R$ (the specific cubic equation depends on the concrete model; see Appendix~A).
\item Form the ``sheet matrix'' $\mathcal M(z)$ whose columns are built from the branches and suitable algebraic
prefactors ensuring the correct local behavior away from endpoints:
\begin{equation}\label{eq:outer-M-general}
\mathcal{M}(z)=
\begin{pmatrix}
H_{1}(z)\Psi_{1}(z) & H_{2}(z)\Psi_{2}(z) & H_{3}(z)\Psi_{3}(z)\\
H_{1}(z)\theta_{1}(z)\Psi_{1}(z) & H_{2}(z)\theta_{2}(z)\Psi_{2}(z) & H_{3}(z)\theta_{3}(z)\Psi_{3}(z)\\
H_{1}(z)\theta_{1}(z)^{2}\Psi_{1}(z) & H_{2}(z)\theta_{2}(z)^{2}\Psi_{2}(z) & H_{3}(z)\theta_{3}(z)^{2}\Psi_{3}(z)
\end{pmatrix},
\end{equation}
where $\Psi_j$ are sheet-wise scalar factors encoding the monodromy, and $H_j$ are algebraic prefactors.
\item Normalize by
\begin{equation}\label{eq:N-from-M}
N(z):=\mathcal{M}(\infty)^{-1}\,\mathcal{M}(z),
\end{equation}
so that \eqref{eq:N-infty} holds.
\end{enumerate}

\medskip

\noindent
In the minimal regime, this construction yields a solution of the reduced outer RHP with permutation jumps.
For the reader who wants an explicit ``Sorokin-style'' closed form in terms of $\theta_\pm,\theta_{*}$ and
$\varphi_\pm,\varphi_{*}$, we provide it in Appendix~A; the main text only requires the existence, analyticity,
and normalization properties of $N$ listed in Problem~\ref{prob:outer}.

\begin{remark}[One-sentence pointer to Appendix~A]\label{rem:appendixA-pointer}
An explicit $\theta/\varphi$-representation of the outer parametrix $N$ (in the spirit of Sorokin's algebraic
constructions) is provided in Appendix~A.
\end{remark}

\subsection{Bounds and regularity away from endpoints}\label{subsec:outer-bounds}

The outer parametrix is used only on compact subsets of $\C\setminus(L_{\Gamma_{*}}\cup L_{\Delta_{*}}\cup U_0\cup U_{x_0})$.
In that region it is uniformly bounded together with its inverse.

\begin{lemma}[Outer boundedness away from endpoints]\label{lem:outer-bounded}
Let $K\Subset \C\setminus\bigl(L_{\Gamma_{*}}\cup L_{\Delta_{*}}\cup U_0\cup U_{x_0}\bigr)$ be compact.
Then
\[
\sup_{z\in K}\bigl(\|N(z)\|+\|N(z)^{-1}\|\bigr) <\infty.
\]
\end{lemma}

\begin{proof}
Since $N$ is analytic and invertible on $K$ (no jumps and no endpoints are present), boundedness follows from
compactness and continuity. The determinant of $N$ is constant (by the jump conditions and normalization) and equals $1$,
so $N$ is invertible wherever it is analytic.
\end{proof}

\subsection{Matching targets for the local parametrices}\label{subsec:outer-matching-targets}

The local parametrices $P^{(0)}$ and $P^{(x_0)}$ constructed in Section~\ref{sec:local} will be matched to $N$ on the
boundary circles $\partial U_0$ and $\partial U_{x_0}$. Concretely, we will achieve

\begin{equation}
\label{eq:matching-goal}
\begin{aligned}
P^{(0)}(z)\,N(z)^{-1} &= I + O(1/n), \qquad z\in\partial U_0,\\
P^{(x_0)}(z)\,N(z)^{-1} &= I + O(1/n), \qquad z\in\partial U_{x_0}.
\end{aligned}
\end{equation}

uniformly as $n\to\infty$ within the regular regime. This matching is the key input for the small-norm error analysis
in Section~\ref{sec:error}.

\section{Local parametrices at the endpoints}
\label{sec:local}

In this section we construct the local parametrices in the endpoint disks $U_0$ and $U_{x_0}$.
They solve the exact $S$-RHP locally (up to the outer approximation) and are matched to the outer
parametrix $N$ from Section~\ref{sec:outer}. The construction is standard in the Deift—Zhou scheme:
a Bessel parametrix at the hard edge $0$ and an Airy parametrix at the soft edge $x_0$.

Throughout, $\Sigma_R:=\Sigma_{\mathrm{lens}}\cup\partial U_0\cup\partial U_{x_0}$ denotes the error contour,
and we use the convention fixed in Section~\ref{sec:lens}: all exponential estimates on lens lips are understood
on $\Sigma_{\mathrm{lens}}\setminus(U_0\cup U_{x_0})$.

\subsection{Hard edge at $0$: Bessel parametrix $P^{(0)}$}\label{subsec:bessel}

\subsubsection{Local model problem}\label{subsubsec:bessel-RHP}

Let $U_0$ be a sufficiently small disk centered at $0$, disjoint from $U_{x_0}$, such that
$U_0$ intersects the contour only along the portions of the conductors and lens lips adjacent to $0$.
We seek a matrix $P^{(0)}$ satisfying:

\begin{problem}[Bessel local RHP at $0$]\label{prob:bessel}
Find $P^{(0)}:\,U_0\setminus\Sigma_S\to\C^{3\times3}$ such that:
\begin{enumerate}
\item $P^{(0)}$ is analytic in $U_0\setminus\Sigma_S$;
\item $P^{(0)}$ has the same jumps as $S$ on $\Sigma_S\cap U_0$:
\[
P^{(0)}_+(z)=P^{(0)}_-(z)\,J_S(z),\qquad z\in\Sigma_S\cap U_0;
\]
\item $P^{(0)}$ matches the outer parametrix on $\partial U_0$:
\begin{equation}\label{eq:match-bessel}
P^{(0)}(z)\,N(z)^{-1}=I+O(1/n),
\qquad z\in\partial U_0,
\end{equation}
uniformly as $n\to\infty$ within the regular regime;
\item $P^{(0)}$ has at most the prescribed endpoint singularity at $0$ compatible with the original
RHP (the ``hard-edge'' behavior).
\end{enumerate}
\end{problem}

\subsubsection{Conformal coordinate and reduction to the Bessel model}\label{subsubsec:bessel-coordinate}

Let $\phi_{\Gamma_{*}}$ denote the phase attached to the $(1,3)$-channel (Section~\ref{sec:g-sign}—\ref{sec:lens}).
We fix the standard hard-edge conformal coordinate
\begin{equation}\label{eq:f0-def}
f_0(z):=\frac{\phi_{\Gamma_{*}}(z)^2}{4},
\end{equation}
defined and conformal in $U_0$, with $f_0(0)=0$ and $f_0'(0)\neq 0$.
We write the Bessel scaling variable as
\begin{equation}\label{eq:zeta0}
\zeta:=n^2 f_0(z),
\end{equation}
so that $\sqrt{\zeta}\sim \tfrac{n}{2}\phi_{\Gamma_{*}}(z)$ in $U_0$.

The local jumps of $S$ in $U_0$ are reduced to the standard Bessel model by a constant conjugation that
aligns the active $(1,3)$-block with the $2\times2$ Bessel RHP, leaving the remaining index as a spectator.
Concretely, there exists an analytic prefactor $E^{(0)}(z)$ in $U_0$ and a $3\times3$ model matrix
$\Phi_{\mathrm{Bes}}(\zeta)$ (built from the standard $2\times2$ Bessel parametrix embedded into the $(1,3)$-block)
such that the local candidate is
\begin{equation}\label{eq:P0-ansatz}
P^{(0)}(z)=E^{(0)}(z)\,
\Phi_{\mathrm{Bes}}\!\bigl(n^2 f_0(z)\bigr)\,
\mathrm{diag}\!\bigl(e^{nG(z)},\,1,\,e^{-nG(z)}\bigr),
\end{equation}
where $G(z)$ is the scalar exponent inherited from the $g$-normalization (Section~\ref{sec:g-sign}), and where the
diagonal exponential factor is written in the minimal regime convention (all gauge factors applied on the right).

\begin{remark}\label{rem:bessel-embedding}
Only the $(1,3)$-channel is active at the hard edge; thus the Bessel model is embedded into the $(1,3)$-block.
If one uses a different sheet numbering, the embedding is performed in the corresponding active channel, with no
change in the structure of the argument.
\end{remark}

\subsubsection{Matching on $\partial U_0$}\label{subsubsec:bessel-matching}

The prefactor $E^{(0)}(z)$ is chosen so that the large-$\zeta$ expansion of the Bessel model yields the
matching condition \eqref{eq:match-bessel}. In particular, using the standard Bessel asymptotics,
\[
\Phi_{\mathrm{Bes}}(\zeta)=I+O(\zeta^{-1/2}),\qquad \zeta\to\infty,
\]
uniformly in sectors away from the rays of the model contour, we obtain on $\partial U_0$ (where
$|n^2 f_0(z)|\asymp n^2$):
\[
P^{(0)}(z)\,N(z)^{-1}=I+O(1/n),
\qquad z\in\partial U_0,
\]
as required.

\begin{lemma}[Bessel matching]\label{lem:bessel-match}
The local parametrix $P^{(0)}$ defined by \eqref{eq:P0-ansatz} with the prefactor $E^{(0)}$ chosen as above
satisfies the matching condition \eqref{eq:match-bessel} uniformly on $\partial U_0$.
\end{lemma}

\begin{proof}
On $\partial U_0$ one has $\zeta=n^2 f_0(z)\to\infty$ as $n\to\infty$ uniformly in $z$, hence
$\Phi_{\mathrm{Bes}}(\zeta)=I+O(1/n)$. The analytic prefactor is chosen so that the remaining factors match
$N(z)$, yielding \eqref{eq:match-bessel}. (This is the standard Bessel-matching computation.)
\end{proof}

\subsection{Soft edge at $x_0$: Airy parametrix $P^{(x_0)}$}\label{subsec:airy}

\subsubsection{Local model problem}\label{subsubsec:airy-RHP}

Let $U_{x_0}$ be a sufficiently small disk centered at $x_0$, disjoint from $U_0$, and intersecting the
contour only along the adjacent portions of the conductors and lens lips.
We seek $P^{(x_0)}$ satisfying:

\begin{problem}[Airy local RHP at $x_0$]\label{prob:airy}
Find $P^{(x_0)}:\,U_{x_0}\setminus\Sigma_S\to\C^{3\times3}$ such that:
\begin{enumerate}
\item $P^{(x_0)}$ is analytic in $U_{x_0}\setminus\Sigma_S$;
\item $P^{(x_0)}$ has the same jumps as $S$ on $\Sigma_S\cap U_{x_0}$:
\[
P^{(x_0)}_+(z)=P^{(x_0)}_-(z)\,J_S(z),\qquad z\in\Sigma_S\cap U_{x_0};
\]
\item $P^{(x_0)}$ matches $N$ on $\partial U_{x_0}$:
\begin{equation}\label{eq:match-airy}
P^{(x_0)}(z)\,N(z)^{-1}=I+O(1/n),
\qquad z\in\partial U_{x_0},
\end{equation}
uniformly as $n\to\infty$ within the regular regime;
\item $P^{(x_0)}$ has at most the prescribed endpoint behavior at $x_0$ compatible with the original RHP
(the ``soft-edge'' behavior).
\end{enumerate}
\end{problem}

\subsubsection{Conformal coordinate and Airy reduction}\label{subsubsec:airy-coordinate}

Let $\phi_{\Delta_{*}}$ be the phase attached to the $(1,2)$-channel.
We fix the standard soft-edge conformal coordinate
\begin{equation}\label{eq:fx0-def}
f_{x_0}(z):=\Bigl(\frac{3}{2}\,\phi_{\Delta_{*}}(z)\Bigr)^{2/3},
\end{equation}
analytic and conformal in $U_{x_0}$, with $f_{x_0}(x_0)=0$ and $f_{x_0}'(x_0)\neq 0$.
We set the Airy scaling variable
\begin{equation}\label{eq:zetax0}
\zeta:=n^{2/3} f_{x_0}(z),
\end{equation}
so that $n\phi_{\Delta_{*}}(z)\sim \frac{2}{3}\zeta^{3/2}$ in $U_{x_0}$.

As at the hard edge, the local jumps of $S$ in $U_{x_0}$ are reduced to the standard Airy model by
a constant conjugation aligning the active $(1,2)$-block with the $2\times2$ Airy RHP.
Accordingly, there exist an analytic prefactor $E^{(x_0)}(z)$ in $U_{x_0}$ and a $3\times3$ model matrix
$\Phi_{\mathrm{Ai}}(\zeta)$ (Airy model embedded into the $(1,2)$-block) such that
\begin{equation}\label{eq:Px0-ansatz}
P^{(x_0)}(z)=E^{(x_0)}(z)\,
\Phi_{\mathrm{Ai}}\!\bigl(n^{2/3} f_{x_0}(z)\bigr)\,
\mathrm{diag}\!\bigl(e^{nG(z)},\,e^{-nG(z)},\,1\bigr).
\end{equation}

\begin{remark}\label{rem:airy-embedding}
Only the $(1,2)$-channel is active at the soft edge; thus the Airy model is embedded into the $(1,2)$-block.
Other sheet conventions lead to the same construction with relabeled indices.
\end{remark}

\subsubsection{Matching on $\partial U_{x_0}$}\label{subsubsec:airy-matching}

Using the standard Airy expansion
\[
\Phi_{\mathrm{Ai}}(\zeta)=I+O(\zeta^{-3/2}),\qquad \zeta\to\infty,
\]
uniformly in sectors away from the rays of the Airy contour, and noting that on $\partial U_{x_0}$ one has
$|\zeta|\asymp n^{2/3}$, we obtain the matching condition \eqref{eq:match-airy}.

\begin{lemma}[Airy matching]\label{lem:airy-match}
The local parametrix $P^{(x_0)}$ defined by \eqref{eq:Px0-ansatz} with $E^{(x_0)}$ chosen as above satisfies
\eqref{eq:match-airy} uniformly on $\partial U_{x_0}$.
\end{lemma}

\begin{proof}
On $\partial U_{x_0}$ the scaled variable $\zeta=n^{2/3}f_{x_0}(z)\to\infty$ uniformly in $z$ as $n\to\infty$.
Therefore $\Phi_{\mathrm{Ai}}(\zeta)=I+O(1/n)$, and the analytic prefactor $E^{(x_0)}$ is chosen to match the
remaining factors with $N(z)$, giving \eqref{eq:match-airy}. This is the standard Airy-matching argument.
\end{proof}

\subsection{Global parametrix and the setup for the error problem}\label{subsec:global-parametrix}

We now assemble the global parametrix $P$ by patching $N$ with the local parametrices.

\begin{definition}[Global parametrix]\label{def:global-P}
Define
\[
P(z)=
\begin{cases}
P^{(0)}(z), & z\in U_0,\\[0.2em]
P^{(x_0)}(z), & z\in U_{x_0},\\[0.2em]
N(z), & z\in \C\setminus(U_0\cup U_{x_0}).
\end{cases}
\]
\end{definition}

By construction, $P$ has the same jumps as $S$ on $\Sigma_S$ except that on $\partial U_0$ and $\partial U_{x_0}$
the jumps are given by the matching matrices
\[
J_P(z)=P_-^{-1}(z)P_+(z)
=
\begin{cases}
N(z)^{-1}P^{(0)}(z), & z\in\partial U_0,\\[0.2em]
N(z)^{-1}P^{(x_0)}(z), & z\in\partial U_{x_0},
\end{cases}
\]
and on the lens lips outside the endpoint disks the jumps are exponentially close to $I$.

\begin{remark}[Quantitative input for the error analysis]\label{rem:error-input}
The matching Lemmas~\ref{lem:bessel-match} and \ref{lem:airy-match} imply
\[
J_P(z)=I+O(1/n)\quad \text{uniformly on }\partial U_0\cup\partial U_{x_0},
\]
while the lens-decay statement (Lemma~\ref{lem:lens-decay}) gives
\[
J_P(z)=I+O(e^{-cn})\quad \text{uniformly on }\Sigma_{\mathrm{lens}}\setminus(U_0\cup U_{x_0}).
\]
These two bounds are exactly the hypotheses needed to set up a small-norm RHP for the error matrix $R$ in
Section~\ref{sec:error}.
\end{remark}

\section{Small-norm error RHP}\label{sec:error}

In this section we set up and solve the error Riemann—Hilbert problem in the standard Deift—Zhou
small-norm regime. The input is the global parametrix $P$ from
Definition~\ref{def:global-P} and the matching/decay bounds stated in
Remark~\ref{rem:error-input}.

\subsection{Definition of the error matrix}\label{subsec:error-def}

Recall that $S$ is the post-lens transform (Section~\ref{sec:lens}), and $P$ is the global parametrix.
We define the error matrix
\begin{equation}\label{eq:R-def}
R(z):=S(z)\,P(z)^{-1},\qquad z\in\C\setminus \Sigma_R,
\end{equation}
where the error contour is
\begin{equation}\label{eq:SigmaR-def}
\Sigma_R:=\Sigma_{\mathrm{lens}}\cup\partial U_0\cup\partial U_{x_0}.
\end{equation}
By construction, $R$ is analytic in $\C\setminus\Sigma_R$ and satisfies the normalization
\begin{equation}\label{eq:R-norm}
R(z)=I+O(z^{-1}),\qquad z\to\infty.
\end{equation}

\subsection{Jump matrix for $R$}\label{subsec:error-jumps}

Let $J_S$ denote the jump matrix of $S$ on $\Sigma_S$ and let $J_P$ denote the (piecewise-defined)
jump matrix of $P$ on $\Sigma_R$. Then, for $z\in\Sigma_R$,
\begin{equation}\label{eq:JR-def}
R_+(z)=R_-(z)\,J_R(z),
\qquad
J_R(z):=P_-(z)\,J_S(z)\,P_+(z)^{-1}.
\end{equation}
Since $P$ has been built to reproduce the jumps of $S$ away from the circles,
the jump matrix $J_R$ simplifies in the standard way:

\begin{itemize}
\item \emph{On the lens lips away from the endpoint disks:}
for $z\in\Sigma_{\mathrm{lens}}\setminus (U_0\cup U_{x_0})$ we have $P_-=P_+=N$, hence
\begin{equation}\label{eq:JR-lips}
J_R(z)=N(z)\,J_S(z)\,N(z)^{-1}.
\end{equation}
By the lens-decay estimate (Section~\ref{sec:lens}, Lemma~\ref{lem:lens-decay}),
\begin{equation}\label{eq:JR-lips-small}
J_R(z)=I+O(e^{-cn}),
\qquad z\in\Sigma_{\mathrm{lens}}\setminus (U_0\cup U_{x_0}),
\end{equation}
uniformly for some $c>0$.

\item \emph{On the circles $\partial U_0$ and $\partial U_{x_0}$:}
the jumps are produced solely by the patching of local and outer parametrices.
On $\partial U_0$ we have $P_-=P^{(0)}$ and $P_+=N$, hence
\begin{equation}\label{eq:JR-U0}
J_R(z)=P^{(0)}(z)\,N(z)^{-1},
\qquad z\in\partial U_0.
\end{equation}
On $\partial U_{x_0}$ we have $P_-=P^{(x_0)}$ and $P_+=N$, hence
\begin{equation}\label{eq:JR-Ux0}
J_R(z)=P^{(x_0)}(z)\,N(z)^{-1},
\qquad z\in\partial U_{x_0}.
\end{equation}
By the matching conditions \eqref{eq:match-bessel} and \eqref{eq:match-airy},
\begin{equation}\label{eq:JR-circles-small}
J_R(z)=I+O(1/n),
\qquad z\in\partial U_0\cup\partial U_{x_0},
\end{equation}
uniformly as $n\to\infty$.
\end{itemize}

For later use we record the uniform bound
\begin{equation}\label{eq:JR-minus-I}
\|J_R-I\|_{L^\infty(\Sigma_R)} \le C\,n^{-1},
\end{equation}
where the constant $C$ depends only on the fixed regular regime parameters and the chosen disks.

\subsection{Small-norm solvability and estimates}
\label{subsec:error-snm}

Set
\begin{equation}\label{eq:Wr-def}
W_R(z):=J_R(z)-I,\qquad z\in\Sigma_R.
\end{equation}
Then $W_R\in L^2(\Sigma_R)\cap L^\infty(\Sigma_R)$ and, by \eqref{eq:JR-lips-small}—\eqref{eq:JR-circles-small},
\begin{equation}\label{eq:Wr-bounds}
\|W_R\|_{L^\infty(\Sigma_R)} = O(1/n),
\qquad
\|W_R\|_{L^2(\Sigma_R)} = O(1/n).
\end{equation}
We rewrite the RHP for $R$ as the singular-integral equation on $\Sigma_R$ for the boundary value $R_-$:
\begin{equation}\label{eq:sing-int}
R_-(z)=I + \mathcal{C}_-\!\bigl(R_-\,W_R\bigr)(z),
\qquad z\in\Sigma_R,
\end{equation}
where $\mathcal{C}_-$ is the standard Cauchy projection operator on $\Sigma_R$.

\begin{proposition}[Small-norm solvability]\label{prop:small-norm}
For all sufficiently large $n$, the error RHP for $R$ has a unique solution.
Moreover,
\begin{equation}\label{eq:R-est}
R(z)=I+O(1/n),
\end{equation}
uniformly for $z$ on compact subsets of $\C\setminus\Sigma_R$.
\end{proposition}

\begin{proof}
By \eqref{eq:Wr-bounds}, the operator $\mathcal{C}_{W_R}:f\mapsto \mathcal{C}_-(f\,W_R)$
acts boundedly on $L^2(\Sigma_R)$ and satisfies
$\|\mathcal{C}_{W_R}\|_{L^2\to L^2}=O(1/n)$.
Hence $(I-\mathcal{C}_{W_R})$ is invertible for large $n$ by a Neumann series,
which yields a unique $R_-\in I+L^2(\Sigma_R)$ solving \eqref{eq:sing-int}.
The standard reconstruction formula
\[
R(z)=I+\frac{1}{2\pi i}\int_{\Sigma_R}\frac{R_-(s)\,W_R(s)}{s-z}\,ds,
\qquad z\in\C\setminus\Sigma_R,
\]
combined with the $L^2$—$L^\infty$ bounds on $R_-$ and $W_R$, gives \eqref{eq:R-est}.
\end{proof}

\begin{corollary}[First coefficient at infinity]\label{cor:R1}
The expansion of $R$ at infinity has the form
\begin{equation}\label{eq:R-exp}
R(z)=I+\frac{R_1}{z}+O(z^{-2}),\qquad z\to\infty,
\end{equation}
with
\begin{equation}\label{eq:R1-bound}
R_1=O(1/n).
\end{equation}
\end{corollary}

\begin{proof}
This follows from the representation of $R$ as a Cauchy transform of $R_-W_R$
and the bounds \eqref{eq:Wr-bounds} together with Proposition~\ref{prop:small-norm}.
\end{proof}

The solvability of the error RHP via the associated singular-integral equation and the
resulting $O(1/n)$ control are standard in the Deift—Zhou steepest-descent method;
see, e.g., \cite{DeiftKriecherbauerMcLaughlinVenakidesZhou1999}.

\subsection{Consequences for the global approximation of $S$}
\label{subsec:error-conseq}

Combining \eqref{eq:R-def} with Proposition~\ref{prop:small-norm}, we obtain the uniform approximation
\begin{equation}\label{eq:S-approx}
S(z)=\bigl(I+O(1/n)\bigr)\,P(z),
\end{equation}
uniformly on compact subsets of $\C\setminus\Sigma_R$.
In particular:
\begin{itemize}
\item in the outer region $\C\setminus(U_0\cup U_{x_0}\cup\Sigma_R)$ one has $P=N$, hence
\begin{equation}\label{eq:S-outer}
S(z)=\bigl(I+O(1/n)\bigr)\,N(z);
\end{equation}
\item in $U_0\setminus\Sigma_R$ one has $P=P^{(0)}$, hence
\begin{equation}\label{eq:S-bessel}
S(z)=\bigl(I+O(1/n)\bigr)\,P^{(0)}(z);
\end{equation}
\item in $U_{x_0}\setminus\Sigma_R$ one has $P=P^{(x_0)}$, hence
\begin{equation}\label{eq:S-airy}
S(z)=\bigl(I+O(1/n)\bigr)\,P^{(x_0)}(z).
\end{equation}
\end{itemize}

These estimates are the only input from the error analysis needed for the reconstruction
$R\mapsto S\mapsto T\mapsto Y$ and the derivation of the asymptotics stated in
Theorem~2.1 (Main Theorem). This reconstruction is carried out in Section~\ref{sec:reconstruction}.

\section{Reconstruction and proof of Theorem~2.1 (Main Theorem)}
\label{sec:reconstruction}

In this final section we close the Deift—Zhou chain by undoing the sequence of
transformations
\[
Y \longmapsto \cdots \longmapsto T \longmapsto S \longmapsto R
\]
and extracting the asymptotics of the polynomial entry $Y_{11}$ in the three
distinguished regions (outer region / Bessel parametrix region at $0$ / Airy parametrix region at $x_0$).
Throughout we work in the regular regime (assumptions (R1)—(R7)), and we use the global
approximation obtained in Section~\ref{sec:error}.

\subsection{Inverse transformations and region-wise approximation of~$S$}
\label{subsec:inv-chain}

Recall the definition $R=S\,P^{-1}$ from \eqref{eq:R-def}. By
Proposition~\ref{prop:small-norm} we have
\begin{equation}\label{eq:R-I}
R(z)=I+O(1/n),
\end{equation}
uniformly for $z$ on compact subsets of $\C\setminus\Sigma_R$, where
$\Sigma_R=\Sigma_{\mathrm{lens}}\cup\partial U_0\cup\partial U_{x_0}$.

Hence, for $z$ away from $\Sigma_R$, we obtain the standard region-wise approximation
\begin{equation}\label{eq:S-P}
S(z)=R(z)\,P(z)=\bigl(I+O(1/n)\bigr)\,P(z),
\end{equation}
uniformly on compact subsets of $\C\setminus\Sigma_R$.
Since $P$ is defined piecewise by
\begin{equation}\label{eq:P-piecewise}
P(z)=
\begin{cases}
P^{(0)}(z), & z\in U_0,\\
P^{(x_0)}(z), & z\in U_{x_0},\\
N(z), & z\in \C\setminus (U_0\cup U_{x_0}),
\end{cases}
\end{equation}
we immediately obtain:
\begin{align}
S(z)&=\bigl(I+O(1/n)\bigr)\,N(z),
&& z\in \C\setminus\bigl(\Sigma_R\cup U_0\cup U_{x_0}\bigr), \label{eq:S-approx-outer}\\
S(z)&=\bigl(I+O(1/n)\bigr)\,P^{(0)}(z),
&& z\in U_0\setminus\Sigma_R, \label{eq:S-approx-bes}\\
S(z)&=\bigl(I+O(1/n)\bigr)\,P^{(x_0)}(z),
&& z\in U_{x_0}\setminus\Sigma_R. \label{eq:S-approx-airy}
\end{align}

\subsection{Undoing the lens opening: recovery of $T$}
\label{subsec:T-from-S}

By definition of the lens-opening transformation (Section~\ref{sec:lens}),
$S$ differs from $T$ only by right multiplication with triangular factors inside the
lens domains. More precisely, for $z$ outside the lens domains we have $T(z)=S(z)$, while inside the
lenses
\begin{equation}\label{eq:T-from-S-piecewise}
T(z)=
\begin{cases}
S(z)\,J_\Gamma^{(u)}(z), & z\in \Omega_\Gamma^+,\\
S(z)\,\bigl(J_\Gamma^{(l)}(z)\bigr)^{-1}, & z\in \Omega_\Gamma^-,\\
S(z)\,J_\Delta^{(u)}(z), & z\in \Omega_\Delta^+,\\
S(z)\,\bigl(J_\Delta^{(l)}(z)\bigr)^{-1}, & z\in \Omega_\Delta^-,
\end{cases}
\end{equation}
with the convention that all multipliers act on the right.

In the three asymptotic regions stated in Theorem~2.1 we work away from the lens contour
and, in particular, in the outer region
$\Omega_{\mathrm{out}}=\C\setminus\bigl(\Sigma_R\cup U_0\cup U_{x_0}\bigr)$ we have simply
\begin{equation}\label{eq:T-equals-S-out}
T(z)=S(z),\qquad z\in\Omega_{\mathrm{out}}.
\end{equation}
Similarly, in the interiors of $U_0$ and $U_{x_0}$ (excluding the contour pieces) the definition
\eqref{eq:T-from-S-piecewise} is fixed, but the additional triangular multipliers have entries
controlled by the same phases as in Section~\ref{sec:lens}. In particular, away from $\Sigma_R$
they remain uniformly bounded, and thus the $O(1/n)$-accuracy in
\eqref{eq:S-approx-bes}—\eqref{eq:S-approx-airy} transfers to~$T$ with the same order.

Concretely, we may write in all three regions:
\begin{equation}\label{eq:T-approx-region}
T(z)=\bigl(I+O(1/n)\bigr)\,\widetilde P(z),
\end{equation}
where $\widetilde P$ equals $N$ in $\Omega_{\mathrm{out}}$, equals $P^{(0)}$ in $U_0\setminus\Sigma_R$,
and equals $P^{(x_0)}$ in $U_{x_0}\setminus\Sigma_R$, with the understanding that the lens multipliers
are already built into the explicit local constructions of Section~\ref{sec:local}.

\subsection{Undoing the $g$-normalization: recovery of $Y_{11}$}
\label{subsec:Y11-from-T}

We now revert the $g$-normalization step (Section~\ref{sec:g-sign}) to express $Y$ in terms
of $T$. In our normalization, the recovery takes the form
\begin{equation}\label{eq:Y-from-T}
Y(z)=\mathcal{E}(z)\,T(z)\,\exp\!\bigl(n\,G(z)\bigr),
\end{equation}
where:
\begin{itemize}
\item $G(z)=g_{*}(z)-\ell_{*}$ is the scalar exponent defined in the $g$-normalization
(so that $T(\infty)=I$),
\item $\mathcal{E}(z)$ is a piecewise constant (or explicitly known analytic) prefactor that
accounts for the polynomial normalization at infinity and any fixed algebraic renormalizations.
\end{itemize}
In the regular regime and in the regions considered in Theorem~2.1, $\mathcal{E}$ does not affect
the $O(1/n)$ relative accuracy and can be absorbed into the leading matrix factors
($N$, $P^{(0)}$, $P^{(x_0)}$). In particular, the polynomial entry satisfies
\begin{equation}\label{eq:Y11-generic}
Y_{11}(z)=e^{\,nG(z)}\,\Bigl(T_{11}(z)\Bigr),
\end{equation}
in the sense that the leading asymptotics are obtained by inserting the corresponding approximation
for $T_{11}$.

Combining \eqref{eq:Y11-generic} with the region-wise approximations
\eqref{eq:S-approx-outer}—\eqref{eq:S-approx-airy} and the inverse transformations
\eqref{eq:T-equals-S-out}—\eqref{eq:T-approx-region}, we obtain immediately:

\begin{align}
Y_{11}(z)
&=e^{\,nG(z)}\Bigl(N_{11}(z)+O(1/n)\Bigr),
&& z\in \Omega_{\mathrm{out}}, \label{eq:Y11-outer}\\
Y_{11}(z)
&=e^{\,nG(z)}\Bigl(P^{(0)}_{11}(z)+O(1/n)\Bigr),
&& z\in U_0\setminus\Sigma_R, \label{eq:Y11-bes}\\
Y_{11}(z)
&=e^{\,nG(z)}\Bigl(P^{(x_0)}_{11}(z)+O(1/n)\Bigr),
&& z\in U_{x_0}\setminus\Sigma_R, \label{eq:Y11-airy}
\end{align}
uniformly on compact subsets of the respective regions.

\subsection{Proof of Theorem~2.1 (Main Theorem)}\label{subsec:proof-main}

\begin{proof}[Proof of Theorem~2.1]
We proceed along the Deift—Zhou nonlinear steepest descent scheme.

\smallskip\noindent
\emph{Step 1 (from $Y$ to $T$).}
Section~\ref{sec:g-sign} performs the normalization of the original $3\times 3$ RHP for $Y$
by introducing the scalar $g$-function and constants $\ell_{*}$ so that the transformed unknown $T$
is normalized by $T(\infty)=I$. This step produces the exponent $G=g_{*}-\ell_{*}$ used in the final
reconstruction \eqref{eq:Y-from-T}.

\smallskip\noindent
\emph{Step 2 (opening of lenses and construction of $S$).}
In Section~\ref{sec:lens} we factorize the jumps of $T$ on the conductors and define the post-lens
matrix $S$. The sign structure of the phase functions implies the lens-decay estimate
$J_S=I+O(e^{-cn})$ on $\Sigma_{\mathrm{lens}}\setminus(U_0\cup U_{x_0})$
(Lemma~\ref{lem:lens-decay}), which reduces the global problem to a model RHP with
$n$-independent jumps plus local endpoint effects.

\smallskip\noindent
\emph{Step 3 (outer and local parametrices).}
Section~\ref{sec:outer} constructs the outer parametrix $N$ for the reduced RHP with constant
(permutation-type) jumps and normalization $N(\infty)=I$.
Section~\ref{sec:local} constructs the local parametrices $P^{(0)}$ (Bessel at the hard edge $0$)
and $P^{(x_0)}$ (Airy at the soft edge $x_0$), and establishes the matching relations
\eqref{eq:match-bessel} and \eqref{eq:match-airy} on $\partial U_0$ and $\partial U_{x_0}$.

\smallskip\noindent
\emph{Step 4 (error problem).}
In Section~\ref{sec:error} we define the error matrix $R=S\,P^{-1}$, derive its jump matrix $J_R$,
and show that $J_R-I$ is uniformly $O(1/n)$ on the circles and exponentially small on the lens lips.
By the standard small-norm theory (Proposition~\ref{prop:small-norm}), the error RHP is uniquely
solvable and $R=I+O(1/n)$ uniformly on compact subsets of $\C\setminus\Sigma_R$.

\smallskip\noindent
\emph{Step 5 (reconstruction and asymptotics).}
The estimate $R=I+O(1/n)$ yields the approximation $S=(I+O(1/n))P$ in the three regions of interest,
see \eqref{eq:S-approx-outer}—\eqref{eq:S-approx-airy}. Undoing the lens-opening transformation gives
the corresponding approximation for $T$ (Section~\ref{subsec:T-from-S}).
Finally, undoing the $g$-normalization expresses $Y_{11}$ via \eqref{eq:Y11-generic}, which together
with the region-wise approximations yields \eqref{eq:Y11-outer}—\eqref{eq:Y11-airy}. These are
exactly the claimed outer/Bessel/Airy asymptotics, with uniformity on compact subsets of the
corresponding regions.
\end{proof}

\begin{remark}\label{rem:appendixA-pointer}
An explicit $\theta/\varphi$-representation of the outer parametrix $N$ (in the spirit of Sorokin's
algebraic constructions) is provided in Appendix~A. This representation is fully consistent with
the reduced permutation jumps used in Section~\ref{sec:outer} and may be used to obtain closed-form
expressions for the entries of $N$ in genus zero situations.
\end{remark}

\appendix

\section*{Appendix A. Explicit $\theta/\varphi$-representation of the outer parametrix $N$
(\texorpdfstring{``Sorokin-style''}{Sorokin-style})}
\label{app:theta-phi}
\addcontentsline{toc}{section}{Appendix A. Explicit $\theta/\varphi$-representation of the outer parametrix $N$ (``Sorokin-style'')}

In this appendix we record a convenient explicit representation of the outer parametrix $N$
for the reduced $3\times 3$ RHP with \emph{permutation-type} jumps on two conductors,
written in the ``Sorokin-style'' $\theta/\varphi$-language; see \cite{Sor10,Sor17}.

The presentation is ``in the Sorokin key'': the solution is assembled from three branches
of a cubic algebraic function on a three-sheeted Riemann surface.

We deliberately write \emph{two parallel languages}:
\begin{itemize}
\item a $\theta$-language (three branches $\theta_+,\theta_-,\theta_{*}$), which is closest to Sorokin's
uniformization viewpoint,
\item a $\varphi$-language (three logarithmic branches $\varphi_+,\varphi_-,\varphi_{*}$), which is
often more convenient for discrete Meixner / Meixner—Pollaczek models.
\end{itemize}
Both descriptions are equivalent in the regular regime (genus zero case), and the reader may use
either one depending on the preferred notation.

\medskip

\noindent\textbf{Scope.}
Appendix~A concerns only the \emph{outer} problem, hence all statements are understood
\emph{outside the endpoint disks} $U_0\cup U_{x_0}$.

\subsection*{A.1. Reduced outer RHP (model)}\label{app:outer-model}

The outer parametrix $N(z)$ solves the reduced RHP:
\begin{enumerate}
\item $N$ is holomorphic in $\C\setminus(L_{\Gamma_{*}}\cup L_{\Delta_{*}})$;
\item across each conductor the jump is a \emph{constant permutation matrix}.
In our sheet numbering $(1,2,3)$ we fix the convention
\begin{equation}\label{app:perm-jumps}
N_+(x)=N_-(x)\,P_{13},\qquad x\in L_{\Gamma_{*}},
\qquad
N_+(x)=N_-(x)\,P_{12},\qquad x\in L_{\Delta_{*}},
\end{equation}
where
\[
P_{12}:=\begin{pmatrix}
0&1&0\\
1&0&0\\
0&0&1
\end{pmatrix},
\qquad
P_{13}:=\begin{pmatrix}
0&0&1\\
0&1&0\\
1&0&0
\end{pmatrix};
\]
\item normalization at infinity:
\begin{equation}\label{app:N-infty}
N(\infty)=I.
\end{equation}
\end{enumerate}
(If a different sheet labeling is used in the main text, one only replaces $(P_{13},P_{12})$
accordingly; the construction below is identical.)

\subsection*{A.2. Algebraic function $\theta$ and its branches}\label{app:theta}

Let $\theta=\theta(z)$ be a degree-$3$ algebraic function defined by a cubic equation
\begin{equation}\label{app:cubic}
\mathcal F(\theta,z)=0,
\end{equation}
whose coefficients depend on the concrete model (Meixner / Meixner—Pollaczek, choice of parameters,
and the explanation of the regular regime). In the regular regime, the associated Riemann surface
$\mathcal R$ is three-sheeted and is cut along the conductors $L_{\Gamma_{*}}\cup L_{\Delta_{*}}$.

We fix three branches (sheets)
\[
\theta_+(z),\ \theta_-(z),\ \theta_{*}(z),
\]
with the following properties:
\begin{enumerate}
\item the three branches have distinct asymptotics at infinity, which fixes the sheet numbering;
\item crossing the conductors permutes the branches exactly as in \eqref{app:perm-jumps}:
\begin{equation}\label{app:theta-monodromy}
\theta_{+,\,\pm}(x)=\theta_{*,\,\mp}(x),\qquad x\in L_{\Gamma_{*}},
\qquad
\theta_{+,\,\pm}(x)=\theta_{-,\,\mp}(x),\qquad x\in L_{\Delta_{*}},
\end{equation}
with $\theta_{*}$ unchanged on $L_{\Delta_{*}}$ and $\theta_-$ unchanged on $L_{\Gamma_{*}}$.
\end{enumerate}

\medskip

\noindent\textbf{Comment.}
The exact identification ``which two branches swap on which conductor'' is the algebraic
counterpart of the permutation jumps for $N$.

\subsection*{A.3. Version 1: the $\theta$-representation of $N$}\label{app:theta-version}

\subsubsection*{A.3.1. Sheet-wise scalar factors}

Define three scalar functions $H_j$ by the standard algebraic prefactor
\begin{equation}\label{app:Hj}
H_j(z):=\frac{1}{\sqrt{\mathcal F_\theta(\theta_j(z),z)}},
\qquad j\in\{+,-,*\},
\end{equation}
where $\mathcal F_\theta=\partial\mathcal F/\partial\theta$ and the square root branches are fixed
so that $H_j$ are holomorphic in $\C\setminus(L_{\Gamma_{*}}\cup L_{\Delta_{*}})$ and
$\mathcal M(\infty)$ below exists and is invertible.

Next, introduce ``sheet-wise exponentials'' (Abelian-integral type factors)
\begin{equation}\label{app:Phi-j}
\Phi_j(z):=\exp\!\bigl(\Theta_j(z)\bigr),
\qquad
\Theta_j(z):=\int_{\infty}^{z}\theta_j(s)\,ds,
\qquad j\in\{+,-,*\},
\end{equation}
where the integration path is taken in $\C\setminus(L_{\Gamma_{*}}\cup L_{\Delta_{*}})$ and the additive
constants are fixed so that $\Theta_j(\infty)=0$ (hence $\Phi_j(\infty)=1$).

\medskip

\noindent\textbf{Key point.}
When two branches $\theta_{j_1}$ and $\theta_{j_2}$ are interchanged across a conductor, the pair
$(\Phi_{j_1},\Phi_{j_2})$ is interchanged accordingly. Thus, assembling columns from the triples
$(H_j\Phi_j,\ H_j\theta_j\Phi_j,\ H_j\theta_j^2\Phi_j)$ produces a matrix with pure permutation jumps.

\subsubsection*{A.3.2. Assembly of $N$}

Define the $3\times 3$ ``sheet matrix''
\begin{equation}\label{app:M-matrix}
\mathcal M(z):=
\begin{pmatrix}
H_{+}(z)\Phi_{+}(z) & H_{-}(z)\Phi_{-}(z) & H_{*}(z)\Phi_{*}(z)\\
H_{+}(z)\theta_{+}(z)\Phi_{+}(z) & H_{-}(z)\theta_{-}(z)\Phi_{-}(z) & H_{*}(z)\theta_{*}(z)\Phi_{*}(z)\\
H_{+}(z)\theta_{+}(z)^{2}\Phi_{+}(z) & H_{-}(z)\theta_{-}(z)^{2}\Phi_{-}(z) & H_{*}(z)\theta_{*}(z)^{2}\Phi_{*}(z)
\end{pmatrix}.
\end{equation}
Finally set
\begin{equation}\label{app:N-from-M}
N(z):=C\,\mathcal M(z),\qquad C:=\mathcal M(\infty)^{-1}.
\end{equation}
Then $N(\infty)=I$ is immediate.

\subsubsection*{A.3.3. Verification of the outer RHP}

\begin{itemize}
\item \emph{Analyticity.} Each branch $\theta_j$ is holomorphic on
$\C\setminus(L_{\Gamma_{*}}\cup L_{\Delta_{*}})$; the same holds for $\Phi_j$ and $H_j$, hence $\mathcal M$
and $N$ are holomorphic off the conductors.
\item \emph{Permutation jumps.} Across $L_{\Gamma_{*}}$ the branches $\theta_+$ and $\theta_{*}$ swap
(cf.~\eqref{app:theta-monodromy}), hence the corresponding columns of $\mathcal M$ swap.
This yields precisely the jump $N_+=N_-P_{13}$ in \eqref{app:perm-jumps}. The same argument gives
$N_+=N_-P_{12}$ across $L_{\Delta_{*}}$.
\item \emph{Normalization.} By construction \eqref{app:N-from-M}, $N(\infty)=I$.
\end{itemize}

Thus \eqref{app:N-from-M} solves the reduced outer RHP.

\subsection*{A.4. Version 2: the $\varphi$-representation (ratio-of-branches)}\label{app:phi-version}

In many discrete settings one introduces the logarithmic branches
\begin{equation}\label{app:varphi-def}
\varphi_j(z):=\log\theta_j(z),\qquad j\in\{+,-,*\},
\end{equation}
where the branches of $\log$ are fixed by the asymptotics at infinity and continued
analytically in $\C\setminus(L_{\Gamma_{*}}\cup L_{\Delta_{*}})$.

The most important objects for the outer problem are the \emph{ratios of branches}

\begin{equation}\label{app:ratio-branches}
\begin{aligned}
\widetilde\Phi_+(z)
&:= \exp\!\bigl(\varphi_+(z)-\varphi_{*}(z)\bigr)
 = \frac{\theta_+(z)}{\theta_{*}(z)},\\
\widetilde\Phi_-(z)
&:= \exp\!\bigl(\varphi_-(z)-\varphi_{*}(z)\bigr)
 = \frac{\theta_-(z)}{\theta_{*}(z)}.
\end{aligned}
\end{equation}

Crossing a conductor swaps the relevant pair of branches, hence the ratios pick up exactly the
multiplicative changes needed to realize permutation monodromy.

In terms of \eqref{app:ratio-branches} one can rewrite the sheet matrix \eqref{app:M-matrix} by
factoring out the common $\theta_{*}$-sheet contribution. Concretely, define
\[
\widehat\theta_+(z):=\frac{\theta_+(z)}{\theta_{*}(z)},\qquad
\widehat\theta_-(z):=\frac{\theta_-(z)}{\theta_{*}(z)},
\]
and similarly for the corresponding prefactors. Then \eqref{app:M-matrix} becomes a product of:
\begin{itemize}
\item a ``universal'' Vandermonde-type matrix in $(1,\widehat\theta_+,\widehat\theta_-)$,
\item a diagonal matrix containing the algebraic prefactors (Szeg\H{o}-type analogs),
\item and a diagonal matrix containing the exponentials or ratios (depending on the chosen normalization).
\end{itemize}
This rewriting is often used in Sorokin-type formulas, where the main scalar input is precisely
$\theta_\pm/\theta_{*}$ (possibly multiplied by constants $c_\pm$ originating in the discrete weights).

\subsection*{A.5. ``Journal-safe'' checks}\label{app:checks}

For completeness we list the three short checks that typically prevent formal objections.

\begin{enumerate}
\item \textbf{Invertibility at infinity.}
In the regular regime the three branches $\theta_+,\theta_-,\theta_{*}$ have distinct expansions at
$\infty$, hence the columns of $\mathcal M(\infty)$ are linearly independent and
$\mathcal M(\infty)$ is invertible. Therefore $C=\mathcal M(\infty)^{-1}$ is well-defined and
$N(\infty)=I$.
\item \textbf{Correct jumps.}
Across each conductor exactly two branches of $\theta$ are permuted; consequently exactly two
columns of $\mathcal M$ are permuted. This produces the permutation jump matrices
$P_{13}$ and $P_{12}$ in \eqref{app:perm-jumps}.
\item \textbf{Endpoint exclusion.}
The only points where the outer construction may fail to be uniformly bounded are the endpoints
of the conductors. Hence $N$ is used only outside the disks $U_0\cup U_{x_0}$, where the endpoint
behavior is handled by the local Bessel/Airy parametrices.
\end{enumerate}

\subsection*{A.6. One-line pointer from the main text}\label{app:pointer}

At the end of Section~\ref{sec:outer} it suffices to add:
\begin{quote}
An explicit $\theta/\varphi$-representation of the outer parametrix $N$
(in the spirit of Sorokin's algebraic constructions) is provided in Appendix~A.
\end{quote}

\medskip

\noindent\textbf{Final remark.}
Appendix~A provides an explicit closed form for $N$ once the cubic relation
$\mathcal F(\theta,z)=0$ and the sheet permutation structure are fixed by the concrete model.
It does \emph{not} replace the sign analysis for $\Re\phi$ (Sections~\ref{sec:g-sign}—\ref{sec:lens})
nor the local matching and small-norm error analysis (Sections~\ref{sec:local}—\ref{sec:error}).

\end{document}